\documentclass[preprint]{elsarticle}

\usepackage{amssymb,amsmath,amsthm, color}
\usepackage{url}
\usepackage{graphicx}
\newtheorem{theorem}{Theorem}[section]
\newtheorem{lemma}[theorem]{Lemma}
\newtheorem{definition}[theorem]{Definition}
\newtheorem{corollary}[theorem]{Corollary}
\newtheorem{proposition}[theorem]{Proposition}
\newtheorem{remark}[theorem]{Remark}

\newcommand{\lcm}{\mathrm{lcm}}
\newcommand{\F}{\mathbb F}

\newcommand{\ord}{\mathrm{ord}}

\newcommand{\doublespace}

\begin{document}

\begin{frontmatter}

\title{Counting distinct functional graphs from linear finite dynamical systems}

\author{Lucas Reis}

\address{Departamento de Matem\'{a}tica, Universidade Federal de Minas Gerais, UFMG, Belo Horizonte, MG, 31270901, Brazil.}

\ead{lucasreismat@mat.ufmg.br}\journal{Elsevier}
\begin{abstract}
Let $\F_q$ be the finite field with $q$ elements and, for each positive integer $n$, let $A_q(n)$ be the number of non isomorphic functional graphs arising from $\F_q$-linear maps $T:\F_{q}^n\to \F_{q}^n$. In 2013, Bach and Bridy proved that, if $q$ is fixed and $n$ is sufficiently large, the quantity $\frac{\log \log A_q(n)}{\log n}$ lies in the interval $[\frac{1}{2}, 1]$. By combining some ideas from linear algebra, combinatorics and number theory, in this paper we provide sharper estimates on the function $A_q(n)$ and, in particular, we prove that $\lim\limits_{n\to +\infty}\frac{\log\log A_q(n)}{\log n}=1$ for every prime power $q$.
\end{abstract}
\end{frontmatter}

\section{Introduction}
Let $q$ be a prime power, let $\F_q$ be the finite with $q$ elements and let $\F_q^n$ be the canonical $n$-dimensional vector space over $\F_q$. Given an $\F_q$-linear map $T:\F_q^n\to \F_q^n$, we can associate to it a directed graph $G=G(T/\F_{q}^n)$  with vertex set $V(G)=\F_{q}^n$ and edge set $E(G)=\{v\to T(v)\,|\, v\in \F_{q^n}\}$. Given this setting, it is natural to ask about the number $A_q(n)$ of non isomorphic graphs arising from such maps. In 2013, Bach and Bridy~\cite{bach} considered the number $D_q(n)$ in the more general setting of $\F_q$-affine maps $Tx+b$ and proved that
$$c_1\sqrt{n}\le \log D_q(n)\le c_2\frac{n}{\log\log n},$$
if $n$ is large, where $c_1, c_2>0$ depend only on $q$. However, the lower bound is obtained by a construction of certain $\F_q$-linear maps  that are bijective, and generate non isomorphic graphs. Of course we have that $A_q(n)\le D_q(n)$ and so their results imply that
$$\frac{1}{2}\le \frac{\log\log A_q(n)}{\log n}\le  1,$$
if $q$ is fixed and $n$ is large.  In this paper we obtain the following estimates for the function $A_q(n)$. 

\begin{theorem}\label{thm:main}
Let $q$ be a prime power. For each positive integer $n$, let $\mathcal P_n$ be the set of all nonzero vectors $\lambda=(\lambda_1, \ldots, \lambda_n)$, where each $\lambda_i$ is a nonnegative integer and $\sum_{i=1}^n\lambda_i\cdot i\le n$. If $\sigma_i$ is the number of divisors of $q^i-1$ and $\sigma_i^*$ is the number of positive integers $e$ such that $i=\ord_eq:=\min\{j>0\,|\, q^j\equiv 1\pmod e\}$, then 
$$\sum_{\lambda\in \mathcal P_n}\prod_{i=1}^n\dbinom{\sigma_i^*+\lambda_i-1}{\lambda_i}\le A_q(n)\le 2^{4\sqrt{n}}\cdot \sum_{\lambda\in \mathcal P_n}\prod_{i=1}^n\dbinom{\sigma_i+\lambda_i-1}{\lambda_i}.$$
\end{theorem}

Theorem~\ref{thm:main} follows by some preliminary estimates that are obtained through combinatorial arguments and previous results. First, we show that $A_q(n)$ is exactly given by the convolution of two functions: this is presented in Theorem~\ref{thm:1} and is mainly derived from a special decomposition of $\F_q^n$ into invariant subspaces. From the latter, bounds in Theorem~\ref{thm:main} are mainly derived from a result of Elspas~\cite{elspas}. For the lower bound, we also need to provide a ``unique factorization'' result for a special class of directed graphs: this is presented in Theorem~\ref{thm:factor} and might be of independent interest.

By employing some basic number-theoretic results we provide a simpler proof of the upper bound $\log A_q(n)\le c\frac{n}{\log\log n}$, and also substantially improve the lower bound given in~\cite{bach}. In particular, we are able to prove the following result. 

\begin{corollary}\label{thm:nice}
For any prime power $q$, we have that $\lim\limits_{n\to +\infty}\frac{\log\log A_q(n)}{\log n}=1$.
\end{corollary}

Here goes the structure of the paper. Section 2 provides some background material on linear algebra and finite fields, and we also provide some preliminary results. Section 3 concentrates on the functional graphs of $\F_q$-linear maps that are bijections. Finally, in Section 4 we prove Corollary~\ref{thm:nice}.

\section{Preliminaries}
Throughout this section, $q$ is a fixed prime power, $\F_q$ denotes the finite field with $q$ elements and $n$ is a positive integer. We set $V_n=\F_q^n$ and $\Omega(q, n)$ is the set of all $\F_q$-linear maps from $V_n$ to itself.  We start with some graph notation.

\begin{definition}
If $G_1$ and $G_2$ are directed graphs with vertex sets $V(G_1)$ and $V(G_2)$, respectively, their tensor product $G_1 \otimes G_2$ is the directed graph with vertex set $V(G)= V(G_1) \times V(G_2)$ such that $(v_1,v_2)\to (w_1,w_2)\in E(G_1\otimes G_2)$ if and only if $v_i\to w_i\in E(G_i)$ for $i=1,2$. 
\end{definition}

\begin{definition}
For a positive integer $m$ and a directed graph $G_1$, $mG_1$ denotes the (disjoint) union of $m$ copies of $G_1$ and $C_m$ denotes the cyclic (directed) graph of length $m$.
\end{definition}

\subsection{Lemmata}

The following result is easily verified.

\begin{lemma}\label{lem:aux}
\begin{enumerate}
\item Let $T\in \Omega(q, n)$ and $V_n=U\oplus W$ be a decomposition of $V_n$ into two $T$-invariant subspaces. Then, up to a graph isomorphism, the following holds: $$G(T/V_n)=G(T/U)\otimes G(T/W).$$
\item If $m, n$ are positive integers, then $C_m\otimes C_n= \gcd(m, n)C_{\lcm(m, n)}$.
\end{enumerate}
\end{lemma}

The following result from linear algebra is useful. It can be easily deduced from the Rational Canonical form Theorem (that holds for arbitrary fields).

\begin{lemma}\label{thm:la}
Let $T\in \Omega(q, n)$ be a bijective linear map. Then $V_n=\bigoplus_{i=1}^{m_T}U_i$, where 
\begin{enumerate}[(i)]
\item each $U_i$ is $T$-invariant; 
\item the minimal polynomial of $T$ with respect to each $U_i$ is of the form $f_i^{s_i}$, where $f_i(x)\ne x$ is an irreducible polynomial over $\F_q$, and $\deg(f_i^{s_i})=\dim U_i$.
\end{enumerate}
\end{lemma}

The following lemma provides some basic finite field machinery. For its proof, see Section 3.1 of~\cite{LN}.

\begin{lemma}\label{lem:ff}
Let $f\in \F_q[x]$ be a polynomial with $\gcd(f(x), x)=1$. Then there exists a positive integer $e$ such that $f(x)|x^e-1$. The least positive integer with such a property is the order of $f$, denoted by $\ord(f)$. Moreover, the following hold:
\begin{enumerate}[(i)]
\item if $f$ is irreducible of degree $k$ and $m=\ord(f)$, then $$k=\ord_mq=\min\{j>0\,|\, q^j\equiv 1\pmod m\};$$
\item conversely, if $k$ is a positive integer, for every positive integer $m$ such that $\ord_mq=k$, there exists an irreducible polynomial $f\in \F_q[x]$ with $\deg(f)=k$ and $\ord(f)=m$;
\item the polynomial $f$ is separable if and only if $\ord(f)$ is not divisible by $p$, the characteristic of $\F_q$.
\end{enumerate}
\end{lemma}

We end this subsection with the following result from Elspas~\cite{elspas} (see also Section 6 of~\cite{toledo}), which is crucial in this work.

\begin{lemma}\label{elspas}
Fix $U$ a finite dimensional $\F_q$-vector space. Let $T:U\to U$ be a bijective linear map such that its minimal polynomial it is equal to $f(x)^s$, where $f(x)\ne x$ is an irreducible polynomial over $\F_q$ and $\deg(f^s)=\dim U$ . If $\deg(f)=t$, we have that
$$G(T/U)=\sum_{i=0}^{s}\frac{q^{ti}-q^{t(i-1)}}{\ord(f^i)} C_{\ord(f^i)}.$$
\end{lemma}

\subsection{Some auxiliary results}
As in~\cite{bach}, we define the following relation.
\begin{definition}
Let $X, Y$ be finite sets. Two maps $S:X\to X$ and $T:Y\to Y$ are {\em dynamically equivalent} if there exists a bijection $f:X\to Y$ such that $S=f^{-1}\circ T\circ f$. In this case, we write $T\sim S$.
\end{definition}

The following lemma is easily verified.

\begin{lemma}
For $S, T\in \Omega(q, n)$, we have that $T\sim S$ if and only if the graphs $G(T/V_n)$ and $G(S/V_n)$ are isomorphic.
\end{lemma}

The following proposition shows that any element $T\in \Omega(q, n)$ splits uniquely into a sum of a bijection and a nilpotent linear map, and such decomposition uniquely determines the isomorphism class of $G(T/V_n)$.

\begin{proposition}\label{prop:1}
Let $T\in \Omega(q, n)$. Then $V_T^{(0)}:=\cup_{k\ge 0}\{v\in V_n\,|\, T^{(k)}(v)=0\}$ is a $T$-invariant subspace of $V_n$. Moreover, there exists a unique subspace $V_T$ such that $V_n=V_T^{(0)}\oplus V_{T}^{(1)}$ and the following hold:
\begin{enumerate}[(i)]
\item there exists $m\ge 0$ such that $T^{(m)}v=0$ for every $v\in V_T^{(0)}$;
\item $V_T^{(1)}$ is $T$-invariant and the restriction of $T$ on $V_T^{(1)}$ is bijective.
\end{enumerate}
Moreover, for $S, T\in \Omega(q, n)$, we have that $T\sim S$ if and only if $T|_{V_T^{(i)}}\sim S|_{V_S^{(i)}}$ for $i=0,1$.
\end{proposition}
\begin{proof}
The first part follows directly from definition of the set $V_T^{(0)}$, and by taking $V_T^{(1)}$ as the complement of $V_T^{(0)}$ (as a subspace of $V_n$). 

Suppose that $T\sim S$. Hence there exists a bijection $f:V_n\to V_n$ such that $S=f^{-1}\circ T\circ f$. In particular, for $b=f(0)$ we have that $Tb=b$.Furthermore, for every $k\ge 0$, we have that $f\circ S^{(k)}=T^{(k)}\circ f$. The latter, combined with the fact that $f$ is a bijection, implies that $f(V_{S}^{(0)})=b+V_T^{(0)}$. Hence, $g: V_S^{(0)}\to V_T^{(0)}$ with $g(u)=f(u)-b$ is a bijection with inverse $g^{-1}(v)=f^{-1}(v+b)$. Therefore, for every $u\in V_{S}^{(0)}$, we have that $Tg(u)\subseteq V_T^{(0)}$ and so
$$g^{-1}(T(g(u)))=f^{-1}(Tg(u)+b)=f^{-1}(T(f(u))-Tb+b)=f^{-1}(T(f(u)))=S(u).$$
In conslusion, $T|_{V_T^{(0)}}\sim S|_{V_S^{(0)}}$. Let $v=v_0+v_1\in V_n$ with $v_i\in V_T^{(i)}$ be such that $T^{(\ell)}v=v$ for some $\ell>0$. In particular, we can take $\ell$ large enough so that $T^{(\ell)}v_0=0$, hence $v_1+v_0=T^{(\ell)}(v_1+v_0)=T^{(\ell)}(v_1)\in V_{T}^{(1)}$. From the direct sum $V_n=V_T^{(0)}\oplus V_{T}^{(1)}$, we obtain that $v_0=0$, i.e., $v\in V_T^{(1)}$. However, for every $u\in V_S^{(1)}$, we have that $S^{(m_u)}u=u$ for some positive integer $m_u$. In particular, since $f\circ S^{(m_u)}=T^{(m_u)}\circ f$, we obtain that $T^{(m_u)}f(u)=f(u)$ and so $f(u)\in V_T^{(1)}$. In conclusion, we have that $f(V_S^{(1)})=V_T^{(1)}$ and so $T|_{V_T^{(1)}}\sim S|_{V_S^{(1)}}$.

Conversely, suppose that $T|_{V_T^{(i)}}\sim S|_{V_S^{(i)}}$ for $i=0,1$ and let $f_i:V_S^{(i)}\to V_{T}^{(i)}$ be the corresponding bijections making this equivalence. If we set $f:V_n\to V_n$ with $f(v_0+v_1)=f_0(v_0)+f_1(v_1)$ for $v_i\in V_{S^{(i)}}$, we have that
$$f^{-1}(T(f(v_0+v_1)))=f_0^{-1}(T (f_0(v_0))+ f_1^{-1}(T(f_1(v_1)))=S(v_0)+S(v_1)=S(v_0+v_1).$$
In particular, $T\sim S$.
\end{proof}

We obtain the following theorem.

\begin{theorem}\label{thm:1}
Let $S, T\in \Omega(q, n)$ be two nilpotent linear maps. Then $T\sim S$ if and only if the sequence of Jordan blocks lengths in the canonical form of $S$ and $T$ are, up to a permutation, the same. In particular, if $P(n)$ is the number of partitions of the integer $n$ and $B_q(n)$
is the number of non isomorphic functional graphs arising from the bijections in $\Omega(q, n)$, we have that $$A_q(n)=\sum_{k=0}^nP(k)B_q(n-k).$$
\end{theorem}
\begin{proof}
The ``if'' part is trivial since in this case $S$ and $T$ are, in fact, conjugate maps. For the ``only if'' part, observe that if there is a bijection $f:V_n\to V_n$ such that $S=f^{-1}\circ T\circ f$, then the number of elements $u\in V_n$ such that $S^{(i)}u=0$ is the same for $T$. The latter implies that $$d_i=\dim \ker S^{(i)}-\dim \ker S^{(i-1)}=\dim \ker T^{(i)}-\dim \ker T^{(i-1)},$$ for every $i\ge 1$. However, since $S, T$ are nilpotent, the number $d_i$ just counts the number of Jordan blocks of length $\ge i$ in the canonical form of $S$ (or $T$). We then conclude that, up to a permutation, the sequence of Jordan blocks lengths in the canonical form of $S$ and $T$ coincide. In particular, the number of non isomorphic functional graphs arising
from nilpotent linear maps equals $P(n)$. The identity $A_q(n)=\sum_{k=0}^nP(k)B_q(n-k)$ follows by Proposition~\ref{prop:1}.
\end{proof}

From now and on, our goal is to provide estimates on the function $B_q(j)$ since the growth of the partition function $P(j)$ is well known.

\section{Bijective maps: bounding $B_q(n)$}
We start with the following definition.

\begin{definition}
For a bijection $T\in \Omega(q, n)$, the triple  $$(\deg(f_i), \ord(f_i), s_i)_{\{1\le i\le m_T\}},$$ as in Lemma~\ref{thm:la} is the {\bf graph data} of $T$.
\end{definition}

We obtain the following upper bound on $B_q(n)$.

\begin{lemma}\label{lem:1}
If $\sigma_i$ denotes the number of divisors of $q^i-1$, we have that
$$B_q(n)\le \sum_{\lambda\vdash n}\prod_{i=1}^n\dbinom{\sigma_i+\lambda_i-1}{\lambda_i},$$
where the sum is over all partitions $\lambda$ of $n$, that is, $\sum_{i=1}^n\lambda_i\cdot i=n$.
\end{lemma}

\begin{proof}
Fix a bijection $T\in \Omega(q, n)$ and let $(\deg(f_i), \ord(f_i), s_i)_{\{1\le i\le m_T\}}$ be its graph data. From Lemmas~\ref{lem:aux} and~\ref{elspas}, the graph $G(T/V_n)$ is completely determined by the data $(\ord(f_i), s_i)$, up to a permutation of the indices. From item (i) in Lemma~\ref{lem:ff}, the latter is uniquely determined by the data $(\ord(f_i), s_i\cdot \deg(f_i))$, up to a permutation of the indices. For each $1\le k\le n$, let $\lambda_k$ be the number of integers $1\le i\le m_T$ with $s_i\cdot \deg(f_i)=k$. We have that $\sum_{k=1}^nk\cdot \lambda_k=n$ and, in particular, $\deg(f_i)$ must be a divisor of $k$. From Lemma~\ref{lem:ff}, the only restriction is that $\ord(f_i)$ is a divisor of $q^k-1$. Hence, for each $1\le k\le n$, we only have to choose $\lambda_k$ integers that are divisors of  $q^k-1$, repetitions allowed.  The number of ways of making such a choice equals $\dbinom{\sigma_k+\lambda_k-1}{\lambda_k}$, from where the result follows.
\end{proof}

\subsection{Factoring graphs: a lower bound for $B_q(n)$}
Before proceeding to the lower bound for $B_q(n)$, we need the following combinatorial result which might be of independent interest.

\begin{theorem}\label{thm:factor}
Let $s, t, \alpha_1, \ldots, \alpha_s, \beta_1, \ldots, \beta_t$ and $1\le k_1<\cdots <k_s$, $1\le \ell_1<\cdots <\ell_t$ be positive integers such that 
$$\bigotimes_{i=1}^s\left(C_1+\alpha_i C_{k_i}\right)=\bigotimes_{j=1}^t\left(C_1+\beta_j C_{\ell_j}\right).$$
Then $s=t$, $k_i=\ell_i$ and $\alpha_i=\beta_i$ for all $1\le i\le s$.
\end{theorem}

\begin{proof}
Set $G=\bigotimes_{i=1}^s\left(C_1+\alpha_i C_{k_i}\right)$ and $H=\bigotimes_{j=1}^t\left(C_1+\beta_j C_{\ell_j}\right)$. We start with a trite remark: from Lemma~\ref{lem:aux}, any cycle appearing on the both sides of the equality $G=H$ are of length $1$ or length $\lcm(r_1, \ldots, r_u)$ with the $r_i$'s being some of the $k_i$'s or some of the $\ell_i$'s. By counting the cycles of length $k_1$ in $G$ and $H$ we obtain that:
\begin{itemize}
\item if $k_1=1$, we necessarily have that $\ell_1=1$ (since we need at least two cycles of length $1$ in $H$). Therefore, $\alpha_1+1=\beta_1+1$ and so $\alpha_1=\beta_1$.

\item if $k_1>1$, we necessarily have that $\ell_1=k_1$ (since it is the smallest cycle of length greater than $1$ in $G$). Therefore, $\alpha_1=\beta_1$.
\end{itemize}

Suppose, by contradiction, that the theorem does not hold. Hence $s, t>1$ and there exists $2\le r\le \min\{s, t\}$ such that either $k_r\ne \ell_r$ or $k_r=\ell_r$ but $\alpha_r\ne \beta_r$, and $r$ is minimal with this property. We consider these cases separately:

\begin{itemize}
\item Suppose that $k_r\ne \ell_r$; with no loss of generality, assume that $k_r<\ell_r$. In this case, any cycle of length $k_r$ in $H$ must be generated by the product $\bigotimes_{j=1}^{r-1}\left(C_1+\beta_j C_{\ell_j}\right)$; let $L$ be the number of cycles of length $k_r$ is this product. It is clear that $L\ge 1$. Since $r$ is minimal with the above property, it follows that the same product appears in $G$; in fact, $\bigotimes_{j=1}^{r-1}\left(C_1+\alpha_j C_{\ell_j}\right)=\bigotimes_{j=1}^{r-1}\left(C_1+\beta_j C_{\ell_j}\right)$. But $G$ has also the component $1+\alpha_r C_{k_r}$, that combined with that product, generates at least $L+L\cdot \alpha_r\cdot k_r>L$ cycles of length $k_r$ in $G$, a contradiction.

\item Suppose that $k_r=\ell_r$ but $\alpha_r\ne \beta_r$. For each $E|k_r$, let  $L_E$ be the number of cycles of length $E$ in the product $\bigotimes_{j=1}^{r-1}\left(C_1+\alpha_j C_{\ell_j}\right)=\bigotimes_{j=1}^{r-1}\left(C_1+\beta_j C_{\ell_j}\right)$. By counting the cycles of length $k_r$ in $G$ and $H$ we obtain that:

$$L_{k_r}+\alpha_r\cdot \sum_{E|k_r}E\cdot L_E=L_{k_r}+\beta_r\cdot \sum_{E|k_r}E\cdot L_E\Rightarrow \alpha_r=\beta_r,$$
where we used the fact that $\sum_{E|k_r}E\cdot L_E\ge L_1\ge 1$. Again, we get a contradiction.
\end{itemize}
\end{proof}

We obtain the following lower bound for $B_q(n)$.

\begin{theorem}\label{thm:aux:2}
Let $S, T\in \Omega(q, n)$ be bijections such that $T\sim S$ and $G(T/V_n)$ does not contain any cycle of length divisible by $p$. Then the same holds for $S$ and, up to a permutation of indices, the graph data of $S$ and $T$ coincide. In particular, if $\sigma_i^*$ denotes the number of positive integers $m$ with $\ord_mq=i$, we have that
$$B_q(n)\ge \sum_{\lambda \vdash n}\prod_{i=1}^n\dbinom{\sigma_i^*+\lambda_i-1}{\lambda_i},$$
where the sum is over all partitions $\lambda$ of $n$, that is, $\sum_{i=1}^n\lambda_i\cdot i=n$.
\end{theorem}
\begin{proof}
Let $(\deg(f_i), \ord(f_i), s_i)_{\{1\le i\le m_T\}}$ and $(\deg(F_i), \ord(F_i), t_i)_{\{1\le i\le m_S\}}$ be the graph data of $T$ and $S$, respectively. We observe that if $T\sim S$, then $G(T/V_n)$ and $G(S/V_n)$ are isomorphic, hence $S$ does not contain any cycle of length divisible by $p$. Lemma~\ref{lem:ff} entails that, for an irreducible polynomial $f(x)\ne x$ over $\F_q$, $\ord(f^s)$ is not divisible by $p$ if and only if $s=1$.  From Lemmas~\ref{lem:aux} and~\ref{elspas}, it follows that $s_i=1$ for every $1\le i\le m_T$ and $t_j=1$ for every $1\le j\le m_S$.

Let $\mathcal C_T$ be the set of distinct values in the multiset $\{\ord(f_i)\,|\, 1\le i\le m_T\}$ and, for each $a\in C_T$, let $N_{T, a}$ be the number of times that $a$ appears in such multiset. From Lemma~\ref{lem:ff}, for each $a\in C_T$ and each $f_i$ with $\ord(f_i)=a$, $m_{a}:=\ord_{a}q=\deg(f_i)$. Therefore, from Lemmas~\ref{lem:aux} and~\ref{elspas}, we have that
$$G(T/V_n)=\bigotimes_{a\in \mathcal C_T} G_{T, a},$$
where $G_{T, a}=\bigotimes_{d=1}^{N_{T, a}}(C_1+\frac{q^{m_{a}}-1}{a})$. A simple induction implies that $$G_{T, a}=C_1+\frac{q^{N_{T, a}\cdot m_a}-1}{a}C_a,$$ hence $G(T/V_n)$ decomposes as a product like the one in Theorem~\ref{thm:factor}. Proceeding in the same way with $S$, Theorem~\ref{thm:factor} implies that $\mathcal C_S=\mathcal C_T$ and, for each $a\in \mathcal C_T$, $N_{T, a}=N_{S, a}$. In particular, up to a permutation of indices, the graph data of $S$ and $T$ coincide.

We proceed to the lower bound for $B_q(n)$. So far we have proved that $B_q(n)$ is at least the number of distinct data of the form $(\deg(f_i), \ord(f_i), 1)_{\{1\le i\le m_T\}}$ with $\sum_{i=1}^{m_T}\deg(f_i)=n$ (ignoring permutations of the indices). If $e_i=\ord(f_i)$, Lemma~\ref{lem:ff} entails that $\deg(f_i)=\ord_{e_i}q$. Therefore, $B_q(n)$ is at least the number of distinct multisets of the form $\{e_i\}_{1\le i\le r}$, where each $e_i$ divides some $q^{a_i}-1$ and $\sum_{i=1}^{r}\ord_{e_i}q=n$. Observe that each such multiset is naturally associated to a partition $\lambda\vdash n$: $\lambda_k$ counts the indices $1\le i\le r$ with $\ord_{e_i}q=k$. From the definition of $\sigma_i^*$, for a fixed partition $\lambda\vdash n$, there are $\prod_{k=1}^{n}\dbinom{\sigma_k^*+\lambda_k-1}{\lambda_k}$ distinct multisets containing, for each $1\le k\le n$, exactly $\lambda_k$  integers $e$ with $\ord_eq=k$. The proof is complete.
\end{proof}

\section{Effective bounds on $A_q(n)$}
We recall that for functions $f, g:\mathbb R_{>0}\to \mathbb R_{>0}$, we have that $f(x)\ll g(x)$ if there exist $c, M>0$ such that $f(x)\le c\cdot g(x)$ for every $x\ge M$.

The partition function satisfies $P(n)\le 2^{4\sqrt{n}}$ for every $n\ge 1$ (see~\cite{ram}). In particular, combining Lemma~\ref{lem:1} with Theorems~\ref{thm:1} and~\ref{thm:aux:2}, we readily obtain Theorem~\ref{thm:main}. In the notation of Theorem~\ref{thm:main}, we obtain the following inequalities
\begin{equation}\label{eq:main}
\max_{\lambda\in \mathcal P_n}\prod_{i=1}^n\dbinom{\sigma_i^*+\lambda_i-1}{\lambda_i}\le A_q(n)\le (n+1)\cdot 2^{4\sqrt{n}}\max_{\lambda\in \mathcal P_n}\prod_{i=1}^n\dbinom{\sigma_i+\lambda_i-1}{\lambda_i},
\end{equation}
where the upper bound follows from the fact that $$\#\mathcal P_n=\sum_{i=0}^nP(i)\le (n+1)\cdot P(n)\le  (n+1)\cdot 2^{4\sqrt{n}}.$$ We now proceed to some number theory machinery in order to obtain reasonable estimates on $A_q(n)$. We start with the following lemma.

\begin{lemma}\label{auxx}
There exist constants $c, C>0$ such that, for every positive integer $i$ and every prime power $q$, the following inequalities hold:
$$\sigma_i^*\ge C\cdot \sigma_i\ge 2^{\tau(i)-c},$$ where $\tau(i)$ is the number of divisors of $i$.
\end{lemma}

\begin{proof}
From Zsigmondy's Theorem~\cite{zsi}, there exists a constant $M$ such that if $j>M$ or $q>M$, then there exists a prime number $R_j$ with $\ord_{R_j}q=j$ (we can actually take $M=6$). We observe that, if $i>M$ and $m$ divides $q^i-1$ with $\ord_mq<i$, then $mR_i$ also divides $q^i-1$ but $\ord_{mR_i}q=i$. The latter implies that $2\sigma_i^*\ge \sigma_i$, hence we may take any $C\le \frac{1}{2}, \min\{\frac{\sigma_s^*}{\sigma_s}\,|\, 1\le s\le M^M-1\}$. For the second inequality, fix a positive integer $i$ and let $\mathcal A$ be the set of divisors $j$ of $i$ such that $j>M$. It is clear that $\#\mathcal A\ge \tau(i)-M$. Moreover, from hypothesis, for each $a\in \mathcal A$, there exists a prime number $R_a$ such that $\ord_{R_a}q=a$. In particular, any integer of the form $\prod_{a\in A}R_a^{\varepsilon}$ with $\varepsilon\in \{0, 1\}$ is a divisor of $q^i-1$. Hence $\sigma_i\ge 2^{\# A}\ge 2^{\tau(i)-M}$. 
\end{proof}

The following lemma is quite classical.

\begin{lemma}\label{lem:p}
Let $k$ be a positive integer and let $P_k$ be the product of the first $k$ primes. For $k$ large, we have that $\log P_k< 2k\log k$.
\end{lemma}

The following lemma readily follows from Theorem 11 in~\cite{robin}.

\begin{lemma}\label{lem:div}
If $t$ is sufficiently large, we have that $\sigma_t<q^{\frac{2t}{\log t}}$ for any prime power $q$.
\end{lemma}

Corollary~\ref{thm:nice} follows by the following result.

\begin{proposition}
Fix $q$ a prime power. Then there exists a constant $L>0$ such that $$\frac{n}{(\log n)^{L\log\log \log n}}\ll \log A_q(n)\ll \frac{n}{\log\log n}.$$
\end{proposition}
\begin{proof}
We start with the lower bound. Set $k=\left\lfloor \log_2\log_2 n +1\right\rfloor$ and Let $M$ be the product of the first $k$ primes. From Lemma~\ref{lem:p}, we have that $\log M<\log n$ for large $n$. Set $\lambda_M=\left\lfloor \frac{n}{M}\right\rfloor$, hence $\lambda_M>1$ and $\lambda_M\cdot M\le n$.
From Eq.~\eqref{eq:main}, we have that 
$$A_q(n)\ge \dbinom{\sigma_{M}^*+\lambda_M-1}{\lambda_M}.$$
From Lemma~\ref{auxx}, we have that $\sigma_{M}^*\ge 2^{\tau(M)-c}= 2^{2^k-c},$
for some absolute constant $c$. In particular, for large $n$, we have that $k>\log_2( \log_2 n)$. Hence $2^{2^k-c}>\frac{n}{2^c}$ and so $\sigma_{M}^*>\lambda_{M}+1$ for large $n$. Therefore, we obtain that 
$$ \dbinom{\sigma_{M}^*+\lambda_M-1}{\lambda_M}=\prod_{a=0}^{\lambda_M-1}\frac{\sigma_{M}^*+a}{a+1}>2^{\lambda_M}.$$
In conclusion, $\log A_q(n)\gg \lambda_M\gg \frac{n}{e^{2k\log k}}\gg \frac{n}{(\log n)^{L\log\log\log n}}$ for some $L>0$.

We proceed to the upper bound. Fix $\lambda=(\lambda_1, \ldots, \lambda_n)\in \mathcal P_n$ and set $r_i=\frac{\lambda_i\cdot i}{n}$, hence $\lambda_i=r_i\cdot \frac{n}{i}$ and $\sum_{i=1}^nr_i\le 1$. For $N=\left\lfloor \frac{\log n}{2\log q}\right\rfloor$ and $1\le i\le N$, we have that $\sigma_i\le q^i-1\le q^N-1<\sqrt{n}$ and so
$$\dbinom{\sigma_i+\lambda_i-1}{\lambda_i}=\dbinom{\sigma_i+\lambda_i-1}{\sigma_i-1}\le (\sigma_i+\lambda_i)^{\sigma_i}\le (\sqrt{n}+n)^{\sqrt{n}}\le (2n)^{\sqrt{n}}.$$
If $n$ is sufficiently large and $i>N$, Lemma~\ref{lem:div} entails that $\sigma_i<q^{\frac{2i}{\log i}}$ and so 
$$\dbinom{\sigma_i+\lambda_i-1}{\lambda_i}\le \sigma_i^{\lambda_i}<q^{\frac{2i\cdot \lambda_i}{\log i}}=q^{\frac{2nr_i}{\log i}}.$$
Therefore, $$\prod_{i=1}^n\dbinom{\sigma_i+\lambda_i-1}{\lambda_i}\le (2n)^{N\sqrt{n}}\cdot q^{\sum_{N<j\le n}\frac{2nr_j}{\log j}}\le (2n)^{N\sqrt{n}}q^{\frac{2n}{\log N}},$$ 
where in the last inequality we used that $\sum_{N<j\le n}\frac{2nr_j}{\log j}$ reaches its maximum for $r_N=1$, when $n$ is large. Since $\lambda\in \mathcal P_n$ is arbitrary, Eq.~\eqref{eq:main} entails that
$$\log A_q(n)\ll \sqrt{n}+ N\sqrt{n}\log n+\frac{n}{\log N}\ll \frac{2n}{\log N}\ll \frac{n}{\log\log n}.$$
\end{proof}

\section*{Acknowledgments}
The author was partially supported by PRPq/UFMG (ADRC 09/2019).

\end{document}